\documentclass[12pt]{amsart}

\usepackage{amsmath,amssymb,amscd,eucal,latexsym}

\makeatletter
\@namedef{subjclassname@2010}{%
  \textup{2010} Mathematics Subject Classification}
\makeatother

\newcommand{\Spec}{\mathop{\mathrm{Spec}}\nolimits}

\newcommand{\TITLE}{Weak N\'{e}ron models for cubic polynomial maps over a
  non-Archimedean field}
\newcommand{\TITLERUNNING}{Weak N\'{e}ron models for cubic polynomials}

\title[\TITLERUNNING]{\TITLE}
\date{}

\author[J.-Y. Briend]{Jean-Yves Briend}
\address{Centre de Math\'ematiques et d'Informatique\\
39, rue Fr\'ed\'eric Joliot-Curie\\
13453 Marseille cedex 13,
France.}
\email{briend@cmi.univ-mrs.fr}

\author[L.-C. Hsia]{Liang-Chung Hsia}
 \address{Department of Mathematics \\
         National Central University \\
         Chung-Li, Taiwan 32054 \\
         R. O. C.}
\email{hsia@math.ncu.edu.tw}

\thanks{The second-named author's research is supported by
  NSC-95-2115-M-008-003 of the National Science Concil of Taiwan.}
\date{\today}

\newtheorem{theorem}{\noindent {\bf Theorem}}[section]
\newtheorem{proposition}{\noindent {\bf Proposition}}[section]

\newtheorem{corollary}{\noindent {\bf Corollary}}[section]

\newtheorem{remark}{\noindent {\bf Remark}}
  {\end{list}}

{\nolinebreak \end{trivlist}}

\newenvironment{question}{\noindent{\bf Question }}{}

\renewcommand{\mathbb}{\mathbf}

\newcommand{\CC}{\mathbb{C}}

\newcommand{\KK}{\mathbb{K}}
\newcommand{\LL}{\mathbb{L}}

\newcommand{\PP}{\mathbb{P}}
\newcommand{\QQ}{\mathbb{Q}}
\newcommand{\RR}{\mathbb{R}}

\newcommand{\ZZ}{\mathbb{Z}}
\newcommand{\kk}{\mathbb{k}}

\newcommand{\cI}{\mathcal{I}}

\newcommand{\cM}{\mathcal{M}}

\newcommand{\cO}{\mathcal{O}}

\newcommand{\cV}{\mathcal{V}}

\newcommand{\cX}{\mathcal{X}}

\mathchardef\Gammait="7100
\mathchardef\Deltait="7101
\mathchardef\Thetait="7102
\mathchardef\Lambdait="7103
\mathchardef\Xiit="7104
\mathchardef\Piit="7105
\mathchardef\Sigmait="7106
\mathchardef\Upsilonit="7107
\mathchardef\Phiit="7108
\mathchardef\Psiit="7109
\mathchardef\Omegait="710A
\DeclareFontFamily{U}{euex}{}
\DeclareFontShape{U}{euex}{m}{n}{ <-7> euex7 <8> euex8 <9> euex9 <10->
  euex10}{} 
\DeclareSymbolFont{eusym}{U}{euex}{m}{n}
\DeclareSymbolFontAlphabet{\eusymb}{eusym}
\DeclareMathSymbol{\intop}{\mathop}{eusym}{"52}
    \def\int{\intop\nolimits}
\DeclareMathSymbol{\produit}{\mathop}{eusym}{"51}
	\def\prod{\produit\nolimits}
\DeclareMathSymbol{\ointop}{\mathop}{eusym}{"48}
    
\DeclareMathSymbol{A}{\mathalpha}{operators}{`A}
\DeclareMathSymbol{B}{\mathalpha}{operators}{`B}
\DeclareMathSymbol{C}{\mathalpha}{operators}{`C}
\DeclareMathSymbol{D}{\mathalpha}{operators}{`D}
\DeclareMathSymbol{E}{\mathalpha}{operators}{`E}
\DeclareMathSymbol{F}{\mathalpha}{operators}{`F}
\DeclareMathSymbol{G}{\mathalpha}{operators}{`G}
\DeclareMathSymbol{H}{\mathalpha}{operators}{`H}
\DeclareMathSymbol{I}{\mathalpha}{operators}{`I}
\DeclareMathSymbol{J}{\mathalpha}{operators}{`J}
\DeclareMathSymbol{K}{\mathalpha}{operators}{`K}
\DeclareMathSymbol{L}{\mathalpha}{operators}{`L}
\DeclareMathSymbol{M}{\mathalpha}{operators}{`M}
\DeclareMathSymbol{N}{\mathalpha}{operators}{`N}
\DeclareMathSymbol{O}{\mathalpha}{operators}{`O}
\DeclareMathSymbol{P}{\mathalpha}{operators}{`P}
\DeclareMathSymbol{Q}{\mathalpha}{operators}{`Q}
\DeclareMathSymbol{R}{\mathalpha}{operators}{`R}
\DeclareMathSymbol{S}{\mathalpha}{operators}{`S}
\DeclareMathSymbol{T}{\mathalpha}{operators}{`T}
\DeclareMathSymbol{U}{\mathalpha}{operators}{`U}
\DeclareMathSymbol{V}{\mathalpha}{operators}{`V}
\DeclareMathSymbol{W}{\mathalpha}{operators}{`W}
\DeclareMathSymbol{X}{\mathalpha}{operators}{`X}
\DeclareMathSymbol{Y}{\mathalpha}{operators}{`Y}
\DeclareMathSymbol{Z}{\mathalpha}{operators}{`Z}
\DeclareMathSymbol{e}{\mathalpha}{operators}{`e}

\def\psh{plurisousharmonique \def\psh{psh}}

\newcommand{\Div}{{\rm Div}}

\renewcommand{\epsilon}{\varepsilon}

\renewcommand{\imath}{{\rm i}}

\newcommand{\lcm}{\operatorname{lcm}}

\renewcommand{\phi}{\varphi}

\newcommand{\build}[3]{\mathrel{\mathop{\kern 0pt #1}\limits_{#2}^{#3}}}
\newcommand{\reduce}[1]{\ensuremath{\widetilde{#1}}}

\newcommand{\ceil}[1]{\ensuremath{\lceil{#1}\rceil}}

\begin{document}
\maketitle
\begin{abstract}
The aim of this note is to give an effective criterion to verify
whether a cubic polynomial over a non-Archimedean field has
a weak N\'{e}ron model or not. 
\end{abstract}

\section{Introduction}

Let $V$ be smooth  variety defined over a discretely valued
non-Archimedean field $\KK$ 
and let $\phi:V\longrightarrow V$ be a morphism on $V$. Assume that
there exists a divisor $E \in \Div(V)\otimes \RR$ and a real number
$\alpha >1$ such that $\phi^{\ast}(E)$ is linear equivalent to
$\alpha E$, Call and Silverman~\cite{CS} showed that there exists a
Weil local height 
function $\hat{\lambda}_{V,E,\phi}$ that plays a role  
analogous to the  N\'{e}ron-Tate local height function on an Abelian
variety. In the case of Abelian varieties, 
the  N\'{e}ron-Tate local height can be 
computed using intersection theory on the N\'{e}ron model
(see~\cite{N, BLR})  of an  Abelian variety in question.
This gives rise to a motivation 
to define an analogous notion for a variety $V$ to which a morphism
$\phi:V\longrightarrow V$ is attached. 
 Such a generalization was proposed in the same paper~\cite{CS} by the
 authors and gives rise to the notion of {\em weak
  N\'{e}ron model} of the couple $(V/\KK,\phi)$ over the ring of integers of
$\KK$ (see also~\cite{BLR}, p.~73 sq.~for
an alternative definition which has nothing to do with the setting
discussed here). They showed that indeed if the pair  $(V/\KK,\phi)$ has a
weak N\'{e}ron model, then the local height
$\hat{\lambda}_{V,E,\phi}$ can also be computed using intersection
theory on the model.  

However, in general a weak N\'{e}ron model for a given pair
$(V/\KK, \phi)$ may not exist. In~\cite{H}, Hsia showed that for a rational map
$\phi:\PP^1\longrightarrow\PP^1$ over $\KK$, the existence of a weak
N\'{e}ron model is closely related to dynamical properties of $\phi$ and more
precisely to the presence of points of the Julia set of $\phi$ inside
$\PP^1(\KK)$. This leads to the question on whether or not one can
effectively determine the existence of a weak N\'{e}ron model for a
given pair $(\PP^1/\KK, \phi)$. In the case of elliptic curves (one
dimensional Abelian variety), the Tate's Algorithm~\cite{Ta} computes, among
other things,  the reduction type of an elliptic curves given by a
Weierstrass equation. Analogous to the situation of elliptic curves,
the question we raised here can be viewed as a 
search for algorithm determining the existence of the weak N\'{e}ron
model for a given pair $(\PP^1/\KK, \phi)$   and computing  
the model when it exists. However, 
for general rational maps $\phi$ on $\PP^1/\KK$, it does not seem clear 
that such  an effective algorithm exists.  On the other
hand, in the case of polynomial maps we think it might be plausible to
have optimistic expectation (see Question~\ref{quest:effectiveness}).   
The aim of this paper is to give a positive answer to this question
for cubic polynomial maps. 

As mentioned above, whether or not a weak
N\'{e}ron model exists for a pair $(\PP^1/\KK,\phi)$ is closely
related to the dynamics induced by the 
action of the given morphism $\phi$ on $\PP^1(\KK)$.
In this note, we consider the dynamics of cubic polynomial
maps over the non-Archimedean field $\KK$. In the classical
theory of dynamical systems, 
dynamical properties of cubic polynomials over $\CC$ have received a
lot of attention since the pioneering work of Bodil Branner and John
Hamal Hubbard in~\cite{BHu}. To understand the parameter space of those
polynomials one is naturally led to study the dynamics of cubic
polynomials over the field of Puiseux series over $\CC$ (or
$\overline{\QQ}$) and this has been carried out by Jan Kiwi in~\cite{K}.

Although the theory of
non-Archimedean dynamical systems has 
its origin in the study of arithmetic problems, Kiwi's work shows that 
non-Archimedean dynamics can be of 
great value in understanding complex dynamics as well.
On the other hand,  over a 
discrete valued field it seems possible that one can
describe every dynamics occurring for a cubic polynomial even though
examples show that they can be very complicated. 
We will be concerned with the determination of the {\em $\KK$-rational
  Julia set} (see Section~2) associated to a given cubic
polynomial $ \phi(z)\in \KK[z]$. As a consequence of our main
result (Theorem~\ref{thm:main2}), we show that  the non-emptiness of the
$\KK$-rational Julia set of $\phi$ 
is closely related to the existence of a $\KK$-rational repelling fixed
point of $\phi$ (see Theorem~\ref{thm:main1}).

The plan of the paper is as follows. In Section~\ref{sec:weak Neron
  model} we recall the definition of a weak N\'{e}ron model for a
given pair $(V/\KK, \phi)$ where we restrict ourselves to the case
$V = \PP^1$. Then, we
state our main results
(Theorem~\ref{thm:main2}~and~Theorem~\ref{thm:main1}).   
Section~\ref{sec:proof} is devoted to the proof of our main
result. In Section~\ref{sec:counterexample} we give examples of
polynomials maps with degree higher than three for which
Theorem~\ref{thm:main1} does not hold.


\section{Weak N\'{e}ron models and Julia sets}
\label{sec:weak Neron model}

Let us introduce some basic notations~:~$\KK$ will denote a field
endowed with a non-Archimedean discrete valuation $v$, which will be
assumed to have $\ZZ$ as value group. We will furthermore assume that
$(\KK,v)$ is  henselian (see~\cite{E} for instance). It is
a well known fact that  for any algebraic extension $\LL$ of $\KK$ the
valuation $v$ has a unique extension to a valuation on $\LL$ (and thus
that the Galois group acts by isometries). We'll use the same notation
$v$ to denote the extension of the valuation to  algebraic closure
of $\KK.$ We will denote by
$\cO_{\KK}$ the valuation ring of  $\KK$ and by $\cM_{\KK}$ its unique
maximal ideal. The residue field is then $\kk=\cO_{\KK}/\cM_{\KK}$.  
We fix a uniformizer $\pi$ of $\KK$
so that $v(\pi)=1$ and we endow $\KK$ with an absolute value $|\cdot|$
associated to $v$ so that $|\pi|<1$. 

Let $\phi\in\KK(z)$ be a rational map. As a map, it acts on
$\PP^1(\KK)$ as well as on $\PP^1(\CC_v)$, where $\CC_v$ is the
completion (with respect to the unique extension of $v$) of an
algebraic closure of $\KK$. On the latter, one can define \emph{the
  Julia set of $\phi$,\/} which we denote by $J_{\phi}$, defined as
the set of points around which the family of iterates of $\phi$ is not
equicontinuous with respect to the chordal metric on $\PP_1(\CC_v)$,
see~\cite{H}. In contrast with the complex case, the Julia set may be
empty, for instance if the polynomial $\phi$ has \emph{good
reduction} as defined in~\cite{MS} (see \S~2.1 for a definition). 
Furthermore, as usually $\KK$ is
far from being algebraically closed, it is often the case that
$J_{\phi}$ is non-empty (and even rather large) while the
$\KK$-rational Julia set $J_{\phi}(\KK):= J_{\phi}\cap\PP^1(\KK)$
is itself empty, meaning that all the complicated dynamical
behavior of $\phi$ takes place outside of $\KK$.

\subsection{Reduction of a morphism and Weak N\'{e}ron model  }
\label{subsec:wnm}
One advantage of working over a non-Archimedean field is that one can
reduce maps modulo the maximal ideal $\cM_{\KK}.$ 
Fixing 
homogeneous coordinates $[x,y]$ on $\PP^1$ over $\KK,$ we may write a
rational map  on $\PP^1$ as 
$\phi([x,y]) = [f(x,y), g(x,y)]$ where $f, g \in \cO_{\KK}[x,y]$ are
homogeneous polynomials in $x,y$
without common divisor.  
Multiplying both homogeneous coordinates by an appropriate  $\lambda\in \KK^{\ast},$ we may
further assume that some coefficient of $f, g$ is a unit of
$\cO_{\KK}.$  Let  $\reduce{\phi} =[\reduce{f}, \reduce{g}],$ where  
$\reduce{f}, \reduce{g}$ denote the reductions of the polynomials $f$
and $g$ by reducing their coefficients modulo the maximal ideal
$\cM_{\KK}$ respectively.  We say that $\phi$ has good reduction
(with respect to the coordinate $[x,y]$) if $\deg(\reduce{\phi}) =
\deg(\phi)$ when $\reduce{\phi}$ is viewed as a morphism on $\PP^1$
over the residue field  $\kk$ (c.f.~\cite[\S~4]{MS}).  In the sequel, we'll
say that a morphism $\phi : \PP^1 \to \PP^1$ over $\KK$ has good
reduction if there exists a possible change of coordinate over $\KK$
such that $\phi$ has good reduction with respect to the new system of
coordinates. Moreover, $\phi$ is said to have {\it
  potential good reduction} if there exists a finite extension $\LL$
over $\KK$ so that after a change of coordinates over $\LL$ the given
rational map $\phi$ has good reduction with respect to the new system
of coordinates.

Even though $\phi$ does not have good reduction, it is possible to
define the reduction of $\phi$ by considering a \emph{weak
  N\'{e}ron model} of $\phi$ 
which is a dynamical analogue of a N\'{e}ron model for an abelian
variety over $\KK.$ As in the arithmetic theory of abelian varieties,  
the notion of  weak N\'{e}ron model of a morphism 
(first introduced by Call and Silverman in~\cite{CS}) is made to study
 the canonical (local) height associated to a morphism $\phi :
V \to V$ on  a smooth, projective variety $V/\KK.$   In the
following, we recall the definition of a weak N\'{e}ron model by
restricting ourselves to the case of rational maps on $\PP^1$ over
$\KK.$ We refer the readers to the paper~\cite{CS} for a general definition.


Let $S = \Spec(\cO_{\KK})$.  We say 
that $(\PP^1/\KK,\phi)$ admits a weak N\'{e}ron model over $S$ if there
exists a smooth, separated scheme $\cV$ of finite type  over $S$
together with a morphism $\Phi:\cV/S\longrightarrow \cV/S$ such that
the following conditions hold~:~
\smallskip
\\
$\phantom{\sum}$
(i) The generic fiber  of $\cV/S$ is isomorphic to $\PP^1$ over $\KK,$
\\
$\phantom{\sum}$
(ii) every point $P\in \PP^1(\KK)$ extends to a section
$\overline{P}:S\longrightarrow \cV,$ 
\\
$\phantom{\sum}$
(iii) the restriction
of $\Phi$ to the generic fiber of $\cV$ is exactly $\phi$. 
\smallskip
\\
To fix the terminologies, we say that a separated scheme $\cV/S$ is a model
of $\PP^1$ (over $S$) if it satisfies~(i); if furthermore it
satisfies~(ii) we say that   it has the extension property for
\'{e}tale points.   
The most basic example of a map admitting a weak N\'eron model is that of a rational map having good reduction, in which case we can take $\cV=\PP^1_S$. 

In general, one can not expect that a weak N\'{e}ron model always
exists for arbitrary rational map $\phi.$ 
The link between the two notions of Julia set and weak N\'{e}ron model was
made by Hsia in~\cite{H} who proved that if the $\KK$-rational Julia set is
non-empty, $(\PP^1/\KK,\phi)$ does not admit any weak N\'{e}ron
model. It
gives thus an easy way of constructing maps without weak N\'{e}ron model,
by choosing one with a sufficiently complicated dynamics, or even just
one with a $\KK$-rational repelling fixed point. 
On the other hand,
some interesting families of rational maps, like Latt\`{e}s maps arising
from isogenies on elliptic curves, have been proven to admit weak
N\'{e}ron models (see~\cite{BH1, BH2}). 
If $\phi$ is a polynomial, by~\cite[Theorem~4.3~and~4.8]{H} we have a better
result~:~$(\PP^1/\KK,\phi)$ admits a weak N\'{e}ron model if and only if
the $\KK$-rational Julia set is empty. Thus, in practice,
one can determine  whether the $\KK$-rational Julia set of a
given polynomial map $\phi$ is empty or not, although it 
can be very complicated. 


\subsection{Statement of the main result}
\label{subsec:statement}
Before we state our main result, let us recall that a fixed point
$z\in\CC_v$ is a point such that $\phi(z)=z$ 
and that $z$ is repelling if $v(\phi'(z))<0$ (equivalently
$|\phi'(z)| > 1$).

\begin{theorem}
  \label{thm:main2}
  Let $\phi\in\KK[z]$ be a polynomial of degree three. Then
  $(\PP^1/\KK,\phi)$ admits a weak N\'{e}ron model if and only if $\phi$
  admits no $\KK$-rational repelling fixed point. 
\end{theorem}

\begin{remark}
It follows from~\cite[Theorem~3.1]{H} that $(\PP^1/\KK,\phi)$ does not
admit a weak N\'{e}ron model if there exists a repelling periodic
point for $\phi$. A repelling periodic point
 is necessarily in the Julia set of $\phi$, so that the $\KK$-rational
 Julia set of $\phi$ is non-empty  if $\phi$ admits a
$\KK$-rational repelling fixed point.
The converse is not {\it a priori} clear as there
could well exist some high period repelling points in $J_\phi(\KK)$ 
that could be hard to detect due to the high degree of the equation
defining them.
\end{remark}

Also, notice that  computing or even
determining non-emptiness of the $\KK$-rational Julia set is often
difficult in general (although it has been carried out
completely in the quadratic case in~\cite{BBP}) but for cubic
polynomials  we prove  
the following theorem giving an easily verifiable criteria for
determining the existence of $\KK$-rational Julia set. 

\begin{theorem}
  \label{thm:main1}
With notations as above, let $\phi\in\KK[z]$ be a
polynomial of degree three. If the $\KK$-rational Julia set 
$J_\phi(\KK)$ of $\phi$ is non-empty then it contains a repelling fixed
point of $\phi.$
\end{theorem}
As a direct consequence of Theorem~\ref{thm:main1}, we have the following.

\begin{corollary}
  \label{cor:fixpt}
  If all  $\KK$-rational fixed points are non-repelling, then
  all periodic points in $\KK$ are non-repelling. 
  \end{corollary}

We fist show that Theorem~\ref{thm:main2} implies
Theorem~\ref{thm:main1}.

\begin{proof}[Proof of Theorem~\ref{thm:main1}]
  As a repelling fixed point is necessarily in the Julia set, we see
  that one direction of the implication is clear. So, let's assume that
  there is no $\KK$-rational repelling fixed point of $\phi$. By
  Theorem~\ref{thm:main2}, $(\PP^1/\KK,\phi)$ admits a weak N\'{e}ron
  model over $S$.  We know that if
  $(\PP^1/\KK,\phi)$ admits a weak  N\'{e}ron model over $S$,  then by
  Theorem~3.3 of~\cite{H}, the family of morphisms
  $\{\phi^i\}_{i=0}^\infty$ is equicontinuous on $\PP^1(\KK)$.
  It follows that the rational Julia set $J_\phi(K)$ is empty as desired. 
\end{proof}

The proof of theorem \ref{thm:main2} will have the following simple corollary :
\begin{corollary}
  \label{cor:fixpt2}
  Let $\phi\in\KK[z]$ be a polynomial of degree three. 
If all the fixed points of $\phi$ (in $\CC_v$) are non-repelling, then
$J_\phi$ is  empty. 
\end{corollary}

\begin{remark}
{\em (1)} We note that Corollary~\ref{cor:fixpt2} is true only for
non-Archimedean dynamics as the example $\phi(z)=z^3+z$ shows~:~its
complex Julia set 
is non empty and it admits a lot of repelling periodic points,
although its only fixed point $0$ is non-repelling.  
\\
{\em (2)} Let us finally note that as a corollary to theorem 3 of B\'ezivin's article \cite{Bez} we get that  for a degree 3 polynomial with coefficients in ${\bf C}_p$ whose ${\bf C}_p$--Julia set is non empty then the Julia set is equal to the closure of the set of repelling periodic points.
\end{remark}

The remaining part of this note is devoted to proving
Theorem~\ref{thm:main2}. We present our proof in the the next
section. In the final section, we give a counterexample of 
Theorem~\ref{thm:main1} and pose a question.

\section{Proof of the main result}
\label{sec:proof}

\subsection{Preliminaries}
\label{subsec:preliminary}
As the theorem is stated in terms of fixed points, we let
$\phi(z)=g(z)+z$ so $g(z)=0$ is the fixed points equation. We will
split the proof in two parts, according to whether $\phi$ admits a
fixed point in $\KK$ or not, in which case the polynomial $g$ is
irreducible over $\KK$. Before dealing with the proof {\it per se,} we
can make three remarks~:~ 
\begin{enumerate}
\item it is not hard to see, from the definition of a weak N\'{e}ron
  model, that if $(\PP^1/\KK,\phi)$ admits a weak N\'{e}ron model, then
  $(\PP^1/\LL,\phi)$ admits one also, where $\LL$ is any unramified
  algebraic extension of $\KK$; 
\item the property of having a weak N\'{e}ron model over $\KK$ or not is
  invariant by conjugacy under a M\"{o}bius map with coefficients in
  $\KK$~:~$(\PP^1/\KK,\phi)$ admits a weak N\'{e}ron model if and only if
  $(\PP^1/\KK, f\circ \phi\circ f^{-1})$ admits one, for one and hence
  all $f\in PGL(2,\KK)$. 
\end{enumerate} 
We can thus assume that $\KK$ is strictly henselian and in
particular that the residue field $\kk$ is algebraically closed. 

We will use the method setup in~\cite{H} to construct weak N\'{e}ron model for
cubic polynomials.  For the convenience of the reader, we sketch it briefly. 
Let $\phi :\PP^1\to \PP^1$ be a give morphism over $\KK$  and take $X_0 :=
\PP^1_S$.
Then $X_0$ is a proper (and smooth) model of
$\PP^1$ over $S$ and $X_0$ has the extension property for
\'{e}tale points by  the valuative criterion for properness
(see~\cite[pp.~95--105]{Ha}). We know that $\phi$ extends at least to
an $S$-rational map $\Phi_0 :\PP^1_S\dashrightarrow  \PP^1_S.$
Now, we proceed inductively. Suppose that we have a separated and smooth 
$S$-model $X_i$ of $\PP^1$ having extension property for \'{e}tale
points and an $S$-rational map
$\Phi_i : X_i \dashrightarrow X_i $ extending $\phi$ 
for  integer $i \ge 0.$  If $\Phi_i$ is an $S$-morphism, then
$(X_i, \Phi_i)$ is a weak N\'{e}ron model for $(\PP^1/K,
\phi).$ Otherwise, the set of points where $\Phi_i$ is not defined is of
codimension 2 in $X_i.$ Hence, 
there are only finitely many closed points on the special fiber
 of $X_i$ where $\Phi_i$ is not defined. We  eliminate the
 indeterminacies by blowing up the closed points where $\Phi_i$ is
 not defined and let  $Y_{i+1}$ be the resulting scheme. Then
 $Y_{i+1}$ is a separated $S$-model of $\PP^1$ and still 
 has the extension property for \'{e}tale points. 
Removing the singular  points of $Y_{i+1}$ yields a new scheme denoted
$X_{i+1}$ which is a smooth, separated  
$S$-model of $\PP^1$ having the extension property for \'{e}tale
points. We consider again the extension map $\Phi_{i+1}
:X_{i+1}\dashrightarrow X_{i+1}$  and test if  any new
indeterminacies occur.
In the case of polynomial maps, either the process
continues indefinitely, in which case the $\KK$-rational Julia set is
not empty, or there is an integer $n \ge 0$ such that the
extension $\Phi_n$ is an $S$-morphism and $(X_n, \Phi_n)$ is a weak
N\'{e}ron model for $(\PP^1/\KK, \phi).$ 

It is a standard fact that one one can eliminate the points of
indeterminacy of a rational map by blowing up a coherent sheaf of
ideals (see for example~\cite[Example~7.17.3]{Ha}). In the proof of
Theorem~\ref{thm:main2}, we shall perform explicit blowups. For that
purpose, we now describe more precisely how one performs blowups to eliminate
indeterminacies of a rational map $\phi$ on $\PP^1$. 
Specifically, let $X$ be a smooth $S$-model of $\PP^1$ having the
extension property for \'{e}tale points and let $S$-rational map
$\Phi : X \dashrightarrow X$ be the extension of $\phi.$ 
Note that as $X$ is a
smooth model of $\PP^1$, the irreducible components of  its special
fiber $\reduce{X}$ is isomorphic to the projective line $\PP_\kk^1$ over
$\kk$ with at most finitely many closed points removed. 
Each point of $\reduce{X}(\kk)$ can be lifted to a point $P \in \PP^1(\KK).$
Hence, we'll write closed point of $\reduce{X}$ as $\reduce{P}$
with $P \in \PP^1(\KK).$ 
Suppose that there's a closed point $\reduce{P}$ on some irreducible component
$Z$ of $\reduce{X}$  where  $\Phi$ is not defined. 
The closed point $\reduce{P}$ as a reduced
closed subscheme of $X$ is locally defined by ideal $\cI\subset
\cO_{\KK}[z]$ which is generated by $\pi$ and $z.$  
This means that $\reduce{P}$ in $\reduce{X}$ has
local coordinate $\reduce{z} = 0.$    
Let $X'\to X$ be the blowing-up of $\reduce{P}$ in $X.$ The
exceptional divisor in $X'$ is thus isomorphic to a projective line
$\PP^1$ over $\kk.$ Let $X_{\pi}'$ denote the subset of $X'$ defined
by the equation $z = \pi z'$ where $z'$ is a local coordinate in $X'.$ 
Following~\cite[\S~3.2]{BLR}, we call $X_{\pi}'$ the {\em dilatation} of
$\reduce{P}$ in $X$ (not to be confused with the term \emph{dilation}  which we shall use below to denote homotheties $z\mapsto \lambda z$). 

Note that $\Phi$ is defined
at the generic point $\eta_Z$ of $Z.$ It follows that its image
$\Phi(\eta_Z)$ is either a generic point or a closed point of some
irreducible component $W$ of $\reduce{X}.$
Let $w$ be a local coordinates 
in a neighborhood of
$\Phi(\eta_Z).$ Then, we may 
represent the rational map $\Phi : X \dashrightarrow X$ locally in terms of
the coordinates $z$ and $w$ so that 
\[
   \phi_w(z) := w\circ \phi(z) = \frac{f(z)}{g(z)}\;\; \text{with $f(z),
     g(z) \in \cO_{\KK}[z]$.} 
\]
Note that   $\reduce{\phi_w}(\reduce{z}) =
\reduce{f}(\reduce{z})/\reduce{g}(\reduce{z})$ represents the rational
map from $Z$ to $W$ over $\kk$. By assumption $\reduce{\phi_w}$ is
not defined at $\reduce{P}.$ It follows that  either $\reduce{f}$ and
$\reduce{g}$ have the common zero $\reduce{z} = 0$ or
$\reduce{\phi_w}(0)$ is equal to a point $\alpha \in W$ where $W$
meets with another component of $\reduce{X}.$ 
Notice that the dilatation $X_\pi'$ of
$\reduce{P}$ has integral points $X'_\pi(\cO_\KK)$ corresponding
bijectively to points of $\PP^1(\KK)$ with coordinate $|z| \le |\pi|.$
The extension of $\Phi$ to $X_\pi'$ amounts to replacing $z$ by $\pi
z'$ on $X_\pi'$. Then, we examine whether or not the extension $\Phi: X'
\dashrightarrow X'$ is well defined on the dilatation $X_\pi'$ until we
attain a model $\cX$ such that the extension $\Phi : \cX
\dashrightarrow \cX$ which we denote by $\Phi$ again,  is an
$S$-morphism.  

After these preliminaries, we are ready to give a proof of
Theorem~\ref{thm:main2}. We split our arguments into two parts
according to whether or not $g$ is irreducible over $\KK.$ We fix an
affine coordinate $z$ on $\PP^1$ so that $\phi(z)$ is a polynomial of
degree 3.  We first deal with the irreducible case in
\S~\ref{subsec:irreducible} below then in \S~\ref{subsec:reducible} we
treat the remaining case and finish the proof. 

\subsection{The irreducible case}
\label{subsec:irreducible}

Let us begin with the case where $g$ is irreducible over $\KK$, that is
when $\phi$ admits no $\KK$-rational fixed point. In this case, we
need to show that $(\PP^1/K, \phi)$ has a weak N\'{e}ron model. 
Conjugating $\phi$ by the dilation $z\mapsto z/\pi^s$
and by taking $s$ large enough we may assume that  $\phi$ is of the following
form (without changing notation for the conjugated map)~:~ 
\[
\phi(z)=\frac{1}{\pi^n} f(z),
\]
with $f(z)=uz^3+a_1z^2+a_2z+a_3\in\cO_{\KK}[z]$, $|u|=1$. 
If $n=0$ then the polynomial has good reduction and
thus it has a weak N\'{e}ron model. Let's assume $n\ge 1$ from now
on.

The reduction  $\reduce{f}$ of $f$ must split over $\kk$ since $\kk$
is algebraically closed. 
Recall that $g(z) = \phi(z) - z$. From this
we get that $g(z)=\pi^{-n} f^*(z)$ where $f^*(z)=f(z)-\pi^n z$
satisfying $\tilde{f^*}=\tilde{f}$. If 
$\tilde{f}$ has a simple root in $\kk$  then so does $\tilde{f^*}$
which, by Hensel's lemma, ensures that $f$ has a fixed point in $\KK$,
which is not the case by hypothesis. We thus have that
$\tilde{f}(z)=\tilde{u}(z-\tilde{\alpha})^3$. Let $\alpha \in
\cO_\KK$ be any lift of $\tilde{\alpha}$~:~$f(\alpha)\equiv 0
\pmod{\pi}$. Conjugating $\phi$ by the translation $z\mapsto
z-\alpha$, we may assume that 
$\tilde{f}(z)=\tilde{u}z^3$ and hence $v(a_i)>0$
for $i=1,2,3$.
Let $f^*(z)=uz^3+a_1^*z^2+a_2^*z+a_3^*$  and $n_i=v(a_i^*)>0$ for
$i=1,2,3$.   

\noindent As a consequence of our normalizations,  the following are true.
\smallskip
\\
$\phantom{\sum}$
(i)  $n_i = v(a_i)$ for $i \ne 2$ and $n_2 = v(a_2 - \pi^n)=
\min\{v(a_2), n\},$ 
\smallskip
\\
$\phantom{\sum}$
(ii)  $n_3=3l+r$, with $r=1$ or $2$;\  $n_2\geq 2l+2r/3$ and $n_1\geq
l+r/3$. 
\smallskip
\\
As (i) is clear, we explain (ii). 
Notice that as $g$ is irreducible over $\KK$,  so is $f^*$.
It follows that every root of 
$f^{\ast}$ has the same valuation. Let $\alpha_i, i = 1, 2, 3$ be the
roots of $f^{\ast}$. Then,
$$
 n_3 = \sum_{i=1}^3 v(\alpha_i) = 3 v(\alpha) \quad \text{with
   $\alpha = \alpha_1$}. 
$$
By assumption  $\KK$ is strictly henselian.  It follows that
$v(\alpha)\not\in \ZZ$. Hence, $n_0$ is of the form as claimed in
(iii) and $v(\alpha_i) = l + r/3$ for $i = 1, 2, 3.$ 
Now, let's observe that
\begin{align*}
 a_1^{\ast} & = - u \sum_{i=1}^3 \alpha_i , \\
 a_2^{\ast} & =  u \sum_{1\le i < j \le 3} \alpha_i \alpha_j .
\end{align*}
Then the  inequalities satisfied by $n_1$ and $n_2$ follow by
applying the valuation $v$ on both sides and the strong triangle
inequality of $v$. 



Before we proceed further, we observe that if 
 $n \ge 2l+2r/3 > 2 l$ then we may  conjugate
$\phi$ by the dilation $z\mapsto \pi^l z$ and get (without changing
notation for $\phi$ conjugated)~:~ 
\[
\phi(z)=\frac{1}{\pi^{n-2l}}\left(uz^3+\pi^{-l}a_1 z^2+\pi^{-2l}a_2z
+\pi^{-3l}a_3\right). 
\]
Notice that the  polynomial $ \pi^{n-2l} \phi(z)$ has all the  coefficients
in $\cO_{\KK}$ and $v(\pi^{-3l}a_3)  = r$.  
So, after conjugation  we may assume that  $l=0$ and $n_3 = r = 1$ or
$2$. Since $n_2 \ge 2 r/3$ ($l = 0$), we easily check that $n_2\geq
n_3$ in this case.

If $ n  <2l+2r/3$, then conjugating $\phi$ by $z\mapsto \pi^k
z$ with $k = [n/2]$ gives~:~ 
\[
\phi(z)=\frac{1}{\pi^{n -
    2k}}\left(uz^3+\pi^{-k}a_1z^2+ \pi^{-2k}a_2 z+\pi^{-3k}a_3\right).  
\]
Similarly, $\pi^{n-2k}\phi(z)$ is a polynomial with coefficients in
$\cO_{\KK}$. If $n = 2 k$ is even, then 
$\deg \reduce{\phi} = \deg \phi$. Thus 
$\phi$ has good reduction in this case and it admits a weak N\'{e}ron
model. On the other hand, if $n$ is odd then $n - 2k =
1$. Hence, after conjugation, we may assume $n = 1$ in this case. 
In our discussion for the remaining case below, we
make further assumption that either (1) $l = 0,$ and $n \ge 2 r/3 $; or
(2) $n = 1 < 2 l + 2 r/3.$

Notice that, $\reduce{z} = 0$ is the only place where the extension of
$\phi$ on $\PP^1_S$ has indeterminacy.  As explained in
\S~\ref{subsec:preliminary}, we  
perform a blowup of the special fiber at $\reduce{z} = 0$ in
$\PP^1_S.$ Let $X_1\to \PP_S^1$ be the blowup and let $\cX$ be the smooth
locus of $X_1.$
\begin{proposition}
  \label{prop:irreducible}
  Let $\phi(z)$  be a cubic polynomial as above. Then, $\phi$ extends
  to a morphism $\Phi : \cX \to \cX$ over $S$ so that 
 $(\cX, \Phi)$ is a weak N\'{e}ron model for $(\PP^1/\KK, \phi)$.
\end{proposition}

\begin{proof}
Let $X_{1,\pi}$ be the dilatation of $\reduce{z} = 0$. Then, on 
$X_{1,\pi}$  we may use affine coordinate
$z_1$ so that  $z=\pi z_1$ and on $X_{1,\pi}$ the polynomial map
$\phi$ can be represented by 
\[
\psi(z_1)=\phi(\pi z_1)=\frac{1}{\pi^n}(u\pi^3
z_1^3+a_1\pi^2z_1^2+a_2\pi z_1+\pi^{n_3}u'), 
\]
where $u'$ is a unit such that $a_3=\pi^{n_3}u'$. 

By our assumption above, we have either (1) $l = 0,$ and $n \ge 2 r/3 $; or
(2) $n = 1 < 2 l + 2 r/3$.  For case (1), we have $n_3 = 1 $ or $2$;
$n\geq n_2\geq n_3$ and $v(a_2) \ge n_2 \ge n_3.$ Then 
\[
\psi(z_1)=\frac{1}{\pi^{n-n_3}}\left(\pi^{3 - n_3} uz_1^3+\pi^{2 -  n_3}a_1 
z_1^2+\pi^{1-n_3} a_2 z_1+u'\right) =
\frac{1}{\pi^{n-n_3}} \psi_1(z_1)
\]
Note that $\psi_1(z_1) \equiv \reduce{u'}\pmod{\pi}. $
Thus $\psi$ sends the component 
$\reduce{X_{1,\pi}}$ on which it is defined to $\tilde{\infty}$ if $n
> n_3$ or
$\tilde{u'}\neq\tilde{0}$ in the special fiber of $\PP^1_S$ if $n = n_3$.
From this,  we conclude that the extension $\Phi
:\cX\to \cX$ is an $S$-morphism. Thus, $(\cX, \Phi)$ is a N\'{e}ron
model for $(\PP^1/\KK, \phi)$ and complete the proof of the first
case. 

Now let's consider case (2). As $2 r/3 > 1$ in this case, we have
$n_3 = 3 l + r \ge 2.$ Therefore,
\[
\psi(z_1)=u\pi^2z_1^3+a_1\pi z_2^2+a_1z_2+ u' \pi^{n_3 - 1} = \pi
\psi_1(z_1)
\]
where $\psi_1(z_1) \in \cO_{\KK}[z_1].$ We see that $\phi$
extends to an $S$-morphism  that maps the component
$\reduce{X}_{1,\pi}$ to itself. In this case, we may also conclude
that $\phi$  extends to an $S$-morphism $\Phi :\cX \to \cX$ and $(\cX,
\Phi)$ is a weak N\'{e}ron model of $(\PP^1/\KK, \phi).$    
\end{proof}

This concludes the case when the fixed points equation is irreducible
over $\KK$. 

\subsection{The reducible case}
\label{subsec:reducible}

We assume in this section that $\phi$ admits a $\KK$-rational fixed
point, which we may assume, after conjugating by a translation, is
equal to $0$. So $\phi$ has the following form~:~ 
\[
\phi(z)=\lambda z+a_2z^2+a_3z^3,\quad \lambda,a_2,a_3\in\KK.
\]
As $\lambda$ is the multiplier of the fixed point $0$, if
$v(\lambda)<0$ then $0$ is a repelling fixed point and $\phi$
does not admit a weak N\'{e}ron model. We thus assume now that
$v(\lambda)\geq 0$. We let $\nu=v(a_3)/2-v(a_2)$ and notice that this
quantity is invariant under conjugacy of $\phi$ under a dilation
centered at $0$. Moreover, by a suitable choice of dilation $z\mapsto
\pi^l z$, we may also assume that $a_2,a_3\in \cO_{\KK}$. As before,
we let $n_i=v(a_i)$ and will distinguish two cases according to the
sign of $\nu$. 

Let's assume first that $\nu\leq 0$, i.e. $n_2 \ge n_3/2.$ 
By conjugating by the dilation $z\mapsto \pi^k z$ with $k = [n_3/2]$, we
may assume that $n_3 = 0$ or $1$. If $n_3 = 0$ then $\phi$ has good
reduction and we're done. So, let's assume that $n_3 = 1$ and take
$X_0 = \PP^1_S$. We see that  $\reduce{\phi} \equiv
\reduce{\lambda}\pmod{\pi}$ and it has an indeterminacy at the point
$\reduce{z} = \infty.$ Let $X_1\to X_0$ be the blowup of the point
$\reduce{z} = \infty$ and $\cX$ be the smooth locus of $X_1.$  Let
$X_{1,\pi}$ be the dilation of $\reduce{z} = \infty$. We may use the
affine coordinate  $z_1 = \pi z$ on $X_{1,\pi}$ and $\phi$ is
represented by 
\[
\psi(z_1)= \phi(z_1/\pi) = \frac{\pi \lambda z_1+a_2z_1^2+ u z_1^3}{\pi^2}
\]
where $a_3 = u \pi$ for some unit $u\in \cO^{\ast}.$ We see that on
$X_{1,\pi}$, $\psi$ has indeterminacy at $\reduce{z_1} = 0$ which is
the point that the special fiber  $\reduce{X_{1,\pi}}$  of
$X_{1,\pi}$ meets with the special fiber of $X_0.$ Therefore, we
conclude that $\phi$ extends to an $S$-morphism $\Phi$ on $\cX$ and
$(\cX, \Phi)$ is a weak N\'{e}ron model for $(\PP^1/\KK, \phi)$ and
prove the case for $\nu \le 0.$ 

If on the contrary we have $\nu >0$ then, by conjugating $\phi$ by
$z\mapsto \pi^{-n_2}z$ we can assume that $n_2=0$ and $n_3=2\nu >0$,
so $\phi$ can be written as 
\[
\phi(z)=\lambda z + u_2z^2+u_3\pi^{2\nu} z^3,
\]
with $u_i\in\cO_{\KK}^*$ and $v(\lambda)\geq 0$. The equation for the
non-zero fixed points is 
\[
u_3\pi^{2\nu}z^2+u_2z+(\lambda-1)=0
\]
which is equivalent, letting $x=\pi^{2\nu}z$, to the equation 
\[
h(x)=u_3x^2+u_2x+(\lambda-1)\pi^{2\nu}=0.
\]
Observe that $h(x)$ has two simple roots modulo $\cM_{\KK}$ so by
Hensel's Lemma $h$ splits over $\KK.$  We conclude that all fixed
points of $\phi$ are $\KK$-rational.  
One of the two roots of $h(x)$ is a unit 
$\epsilon\in\cO_{\KK}^*$ such that $\epsilon\equiv -u_2/u_3
\pmod{\pi}$. Let $\zeta=\epsilon\pi^{-2\nu}$ be the one with
largest absolute value. The multiplier of this fixed point is
given by 
\[
\phi'(\zeta)=\frac{1}{\pi^{2\nu}}
(u_3\epsilon^2+2\epsilon(u_3\epsilon+u_2)+\lambda\pi^{2\nu})=
\frac{g(\epsilon)}{\pi^{2\nu}}
\]
and
\[
\tilde{g}(\tilde{\epsilon})=\tilde{g}(-\frac{\tilde{u_2}}{\tilde{u_3}})=
\frac{\tilde{u_2}^2}{\tilde{u_3}}\neq 0.
\]
Thus $v(\phi'(\zeta))=-2\nu<0$ and $\zeta$ is a repelling fixed
point. So $(\PP^1/\KK,\phi)$ does not admit a weak N\'{e}ron model in
this case, and this concludes the proof of the theorem. 

\begin{remark}
  \label{remark:potential good}
{\em (1)}. Let $\LL$ be the quadratic ramified extension $\KK[\sqrt{\pi}]$ of $\KK$
and denote by $v_{\LL}$ the  extension of the valuation $v$ on
$\LL$ such that $v_{\LL}= 2 v $ on $\KK$ and
$v_{\LL}(\sqrt{\pi})=1$. Equivalently, $v_{\LL}(\LL^{\ast}) = \ZZ.$  
In the case where $\nu \le 0$, we note that $v_{\LL}(a_3) = 2 v (a_3)$
which is an even integer. Then, the same argument for the case
where $n_3$ is even in the proof applies to this situation. In
conclusion,  $(\PP^1/\KK, \phi)$ has potential good reduction if
$\nu\le 0.$
\\
{\em (2)}. Let $\phi(z)\in \KK[z]$ be a cubic polynomial as above.
It follows from the proof of Theorem~\ref{thm:main2} that if 
$(\PP^1/K, \phi)$ has a weak N\'{e}ron model $(\cX, \Phi)$, then the
special fiber $\reduce{\cX}$ of $\cX$ is either $\PP^1_\kk$ in which
case $\phi$ has good reduction; or two components isomorphic to
$\PP^1_\kk$ that intersect at one closed point transversally. In fact,
the proof actually gives an algorithm to compute the reduction type of
$\phi.$ 
\end{remark}

\vspace*{3mm}
\begin{proof}[Proof  of corollary~\ref{cor:fixpt2}:]
  As periodic points
of $\phi$ are algebraic over $\KK$ we may assume $\KK$ contains the
fixed points of $\phi$. We are thus in the situation where the fixed
points equation is reducible. As the fixed points are non-repelling,
with the notations above we are necessarily in the case where $\nu\leq
0$ for which we showed that $\phi$ has potential good reduction by
Remark~\ref{remark:potential good} and thus empty Julia set.  
\end{proof}

\section{Counterexamples  in degree higher than 4}
\label{sec:counterexample}
The purpose of this section is to show that the theorem proved above
is not true any more in degree higher than three. Let the degree $d =
\deg \phi \ge 4$ and  let $p$ be the characteristic of the residue
class field $\kk.$ 
Then, except for the four cases $(p, d) = (2, 4),(2,5),(2, 7),(3,5)$,
we can write $d = e_0 + e_1$ or $d = e_0 + e_1 + e_2$ such that $e_i
\ge 2$ and $p\nmid e_i$ for $i = 0, 1,2.$ For simplicity, we consider the
former case here. That is, $d =  e_0+e_1$ with $e_i\ge 2$ and $p\nmid
e_i$ for $i =0, 1.$ Let $n = \lcm(e_0,e_1)$ and $n = e_0 e_0' = e_1
e_1'.$  
With the same notations as in the previous sections we let 
\[
\phi(z)=\frac{1}{\pi^n}z^{e_0}(z-1)^{e_1}+z
\]
be a degree $d$ polynomial with $0$ and $1$ as fixed
points which are not repelling. 
For ease of notation, we let $(x)_n=\{y\in\KK\nmid v(x-y)\geq n\}$
be the ball centered at $x$ with radius $|\pi |^n$. Let $r_i =
\lceil e_i'/(e_i - 1)\rceil, i = 0,1,$ where $\ceil{x}$ is the
ceiling of the real number $x$ (i.e. the smallest of the integer not
less than $x$). Put $s_i = e_i' + r_i, i = 0,1.$ Then, the two balls
$(0)_{s_0}$ and $(1)_{s_1}$ are both forward invariant (under the action of
$\phi$) and are thus in the Fatou set of $\phi$. A simple computation
shows that if $z$ is not in $(0)_{e_0'}\cup (1)_{e_1'}$ then its orbit escapes
to infinity. The Julia set is thus included in the union of the two
annuli $(0)_{e_0'}\setminus (0)_{s_0}$ and $(1)_{e_1'}\setminus
(1)_{s_1}$. Notice that $\phi$ takes balls $(0)_{e_0'}$ and $(1)_{e_1'}$ onto the unit
ball $(0)_0$. By symmetry, we may just consider  $\phi :(0)_{e_0'} \to
(0)_0$ only.  
Let $z=\pi^{e_0'} w$ so that  $w$ is a local coordinate on $(0)_{e_0'}$. 
Recall that $n = e_0 e_0'$, we get 
\[
\phi(\pi^{e_0'} w)=w^{e_0}(\pi^{e_0'} w -1)^{e_1}+\pi^{e_0'} w\equiv
(-1)^{e_1}w^{e_0}\mod \pi. 
\]
Since $p \nmid e_0$, there are
$\zeta_{i} \in \kk, i = 1,\ldots, e_0$ such that
$$
 (-1)^{e_1} \zeta_i^{e_0} \equiv 1 \pmod{\pi}. 
$$
As a simple consequence of Hensel's lemma, we find that there are
points $a_i \in (0)_{e_0'}$ such that $w(a_i) \equiv \zeta_i
\pmod{\pi}$ and that $\phi : (a_i)_{e_0' + 1} \to (1)_1$ are
bijectively expanding the distances by a factor of
$|\pi|^{-e_0'}$ for all $ i = 1,\ldots, e_0$. 
By the
same reason, there are points $b_j \in (1)_{e_1'}$ such that  $\phi :
(b_j)_{e_1' + 1} \to (0)_1$ are bijectively expanding by a factor of
$|\pi|^{-e_1'}$ for $j = 1,\ldots, e_1.$ We see that $\phi$ induces a
sub-dynamics on
$$
J_{\phi}\bigcap \left(\cup_{i=1}^{e_0} (a_i)_{e_0'+1} \bigcup
  \cup_{j=1}^{e_1} (b_j)_{e_1'+1}\right) 
$$
which can be conjugated
to the subshift of finite type on $e_0 + e_1 = d$ symbols whose
incidence graph is given by the complete bipartite graph with $e_0+e_1$
vertices.  In particular
$\phi$ admits no repelling fixed point while on the other hand it
admits period two repelling points : $\phi$ does not have a weak N\'{e}ron
model over $\cO_{\KK}$. 
For another case where $d = e_0 + e_1+e_2$ with $e_i \ge
2$ and $p\nmid e_i$ one can argue similarly that $J_\phi(K)$ is
non-empty while all the fixed points are non-repelling. 
\medskip

\begin{question}
  \label{quest:effectiveness}
  (1) It is an interesting question to see whether or not
  Theorem~\ref{thm:main1} is true for the four exceptional cases $(p, d) =
  (2,4),(2,5),(2,7)$ and   $(3,5)$. 
  \medskip
  \\
  (2)  Our examples raise the following question: for 
  a polynomial map of degree $d$ is 
  there a positive  integer $r$ depending on $d$ and the
  characteristic $p$ of the residue field $\kk$  so that if all 
  $\KK$-rational periodic  points with period less than $r$ are
  non-repelling  then there is no $\KK$-rational Julia set of the
  polynomial map in question? 
\end{question}
\medskip
\\
{\small\noindent{\bf Aknowlegments:~} the authors of this article
  have benefited from support from the project "Berko" of the french
  Agence Nationale pour la Recherche, of Universit\'e de Provence in
  France and National Center for Theoretical Sciences and National
  Central University in Taiwan.}

\end{document}